\documentclass[oneside,a4paper,11pt,notitlepage]{article}
\usepackage[left=0.8in, right=0.8in, top=1.5in, bottom=1.5in]{geometry}

\usepackage[T1]{fontenc} 
\usepackage[utf8]{inputenc} 
\usepackage[english]{babel}
\usepackage{amssymb}
\usepackage{lipsum} 
\usepackage{lmodern}
\usepackage{amsmath}
\usepackage{amsthm}
\usepackage{bm}
\usepackage{mathtools}
\usepackage{braket}
\usepackage{esint}
\newcommand{\abs}[1]{{\left|#1\right|}}
\newcommand{\norma}[1]{{\left\Vert#1\right\Vert}}

\usepackage{enumitem}
\usepackage{booktabs}
\usepackage{graphicx}
\usepackage{tikz}
\usetikzlibrary{patterns}
\usepackage{multicol}
\usepackage{caption}
\usepackage[skins,theorems]{tcolorbox}
\tcbset{highlight math style={enhanced,
colframe=black,colback=white,arc=0pt,boxrule=1pt}}
\def\XXint#1#2#3{{\setbox0=\hbox{$#1{#2#3}{\int}$}
    \vcenter{\hbox{$#2#3$}}\kern-.5\wd0}}

\theoremstyle{definition}
\newtheorem{definizione}{Definition}[section]
\theoremstyle{plain}
\newtheorem{teorema}{Theorem}[section]

\newtheorem{lemma}[teorema]{Lemma}
\newtheorem{prop}[teorema]{Proposition}

\theoremstyle{definition}
\newtheorem{esempio}{Example}[section]
\newtheorem{oss}[esempio]{Remark}

\newtheorem*{open*}{Open problems}
\DeclareMathOperator{\R}{\mathbb{R}}

\makeatletter
\newcommand{\myfootnote}[2]{\begingroup
	\def\@makefnmark{}%
	\addtocounter{footnote}{-1}%
	\footnote{\textbf{#1} #2}
	\endgroup}
\makeatother

\usepackage{hyperref}
\hypersetup{linktoc=none, bookmarksnumbered, colorlinks=true, linkcolor=magenta, citecolor=cyan}

\title{The asymptotic behavior of the first Robin eigenvalue with negative parameter as $p$ goes to $+\infty$}
\author{Rosa Barbato, Francesca de Giovanni*, Alba Lia Masiello}
\date{}

\newcommand{\Addresses}{{
\bigskip 
  \footnotesize 
 
  \textit{E-mail address}, R.~Barbato: \texttt{rosa.barbato2@unina.it} 
  
   \medskip 

    \textit{E-mail address}, F. ~de Giovanni (corresponding author): \texttt{ francesca.degiovanni@unina.it} 
 
  \medskip
  \textsc{Dipartimento di Matematica e Applicazioni ``R. Caccioppoli'', Universit\`a degli studi di Napoli Federico II, Via Cintia, Complesso Universitario Monte S. Angelo, 80126 Napoli, Italy.}

   \medskip  
   
  \textit{E-mail address}, A.L.~Masiello: \texttt{masiello@altamatematica.it} 
 
   \medskip 
   \noindent \textsc{
 	Holder of a research grant from Istituto Nazionale di Alta Matematica "Francesco Severi" at Dipartimento di Matematica e Applicazioni "R. Caccioppoli", Via Cintia, Complesso Universitario Monte S. Angelo, 80126 Napoli, Italy.}

 \par\nopagebreak 

}} 

\begin{document}

\maketitle
 \begin{abstract}
In this paper, we want to study the asymptotic behavior of the first $p$-Laplacian eigenvalue, with Robin boundary conditions, with negative boundary parameter. In particular, we prove that the limit of the eigenfunctions is a viscosity solution for the infinity Laplacian eigenvalue problem.
\end{abstract}

\section{Introduction}
Let $\Omega\subset \R^n$ be a bounded open set, with Lipschitz boundary, with $n\geq 2$, and let $\beta>0$.
We aim to study the asymptotic behavior as $p$ goes to $+\infty$ of the first eigenvalue of the Robin $p$-Laplacian on $\Omega$ with a negative boundary parameter
\begin{equation}
    \label{rel}
    \lambda_{p, \beta}(\Omega)=\min_{\substack{w\in W^{1,p}(\Omega) \\ w \neq 0}} \frac{\displaystyle{\int_\Omega \abs{\nabla w}^p\, dx-\beta^p \int_{\partial\Omega} \abs{w}^p\, d\mathcal{H}^{n-1}}}{\displaystyle{\int_\Omega \abs{w}^p\, dx}}.
\end{equation}
A minimizer $u$ in \eqref{rel} is a solution to the $p$-Laplace eigenvalue equation with Robin boundary conditions

\begin{equation}
\label{autoval_problem}
\begin{cases}
-\Delta_p u=\lambda_{p,\beta}(\Omega) \abs{u}^{p-2} u & \text{ in } \Omega \\
\abs{\nabla u}^{p-2} \dfrac{\partial u}{\partial \nu} -\beta^p \abs{u}^{p-2}u=0 & \text{ on } \partial \Omega.
\end{cases}
\end{equation}
The eigenvalue $\lambda_{p,\beta}(\Omega)$ is negative, as it can be seen by choosing $w=1$ in the variational formulation \eqref{rel},
$$\lambda_{p,\beta}(\Omega)\leq -\beta^p\dfrac{P(\Omega)}{\abs{\Omega}}.$$
Moreover, as the value of the functional in \eqref{rel} is the same in $w$ and $-w$, we can choose the eigenfunction $u$ to be positive.
 For a negative boundary parameter, some problems are not completely solved, such as the optimization of $\lambda_{p,\beta}(\Omega)$ with a volume constraint. Indeed, in 1977 Bareket \cite{Bareket} conjectured that the first eigenvalue is maximized by a ball in the class of smooth bounded domains of given volume.  The authors in \cite{Freitas_Krejcirik} (see \cite{kovavrik2017p} for the $p$-Laplacian case) have disproved the Bereket conjecture for $\left|\beta\right|^p$ large enough, studying the asymptotic behavior of the eigenvalue in \eqref{rel}, showing that

 $$ \lambda_{p,\beta}(\Omega) = -(p-1)\beta^{\frac{p^2}{p-1}}- (n-1)H_{max}(\Omega)\beta^p+ o(\beta^p), \quad \beta^p \rightarrow +\infty,$$
 
 where $H_{max}(\Omega)$ is the maximum mean curvature of $\partial \Omega$. This asymptotic allows the authors to conclude that the eigenvalue of the spherical shell is larger than the one of the ball with the same measure.
 
 In the same paper, they also showed that the conjecture is indeed true for small values of $\left|\beta\right|^p$ in a suitable class of domains. We also recall that, in the Finsler setting, this problem has been addressed in \cite{paoli2019two}.

In the present paper, we aim to study the limit as $p$ goes to $+\infty$ of problem \eqref{rel}, following the footsteps of \cite{AMNT, AMNT2, EKNT,Lin_ju_lambda2, JLM}. 
It is worth mentioning that the limit case, as $p$ goes to $1$, has been studied in \cite{della2022behavior} for the Euclidean setting, and in  \cite{frangianross} for the anisotropic case. Moreover, the case of higher Robin eigenvalues of the $p$-Laplace operator as $p$ goes to $1$ has been studied in \cite{segura}, also in the case when $\beta$ is a non-negative function in $L^{\infty}(\partial\Omega)$.

In \cite{Lin_ju_lambda2,JLM} the authors study the limit as $p$ goes to $+\infty$ of the eigevalue of the $p$-Laplace operator with Dirichlet boundary conditions, obtaining a complete characterization of the limiting problems in terms of geometric quantities. 
Indeed, the first two eigenvalues of the $p$-Laplace operator, $\lbrace\lambda_{1,p}^D\rbrace$ and $\lbrace\lambda_{2,p}^D\rbrace$, satisfy
$$
\lim_{p\to \infty} \left(\lambda_{1,p}^D\right)^{{1}/{p}} = \lambda_{1, \infty}^D=: \frac{1}{r(\Omega)} \quad\text{ and }\quad \lim_{p\to \infty} \left(\lambda_{2,p}^D\right)^{{1}/{p}} = \lambda_{2,\infty}^D=: \frac{1}{r_2(\Omega)},
$$
where
$$r(\Omega)=\sup\{t: \exists B_t(x)\subseteq\Omega\},$$
\begin{center}
    and
\end{center}
$$r_2(\Omega)=\sup\{t: \exists B_t(x_1), B_t(x_2)\subseteq\Omega \text{ such that } B_t(x_1)\cap B_t(x_2)=\emptyset \}.$$

The Neumann case was investigated in \cite{EKNT, RS2} and, similarly to the Dirichlet case, the authors proved that the first non-trivial eigenvalues of the $p$-Laplacian with Neumann boundary conditions $\lbrace\lambda^N_p \rbrace$ satisfy
$$\lim_{p\to \infty}\left(\lambda^N_p\right)^{1/p} = \lambda_\infty^N := \frac{2}{D(\Omega)},$$
where $D(\Omega)$ is the intrinsic diameter of $\Omega$, i.e. the supremum of the geodetic distance between two points of $\Omega$.

The Robin boundary conditions with \emph{positive} boundary parameter was addressed in \cite{AMNT, AMNT2}, in which the authors focus on the first and the second eigenvalue respectively. As in the previous cases, the geometry of the set plays a key role in determining the limit of the sequence of eigenvalues, as

$$
\lim_{p\to \infty} \left(\lambda_{1,p}^R\right)^{{1}/{p}} = \lambda_{1, \infty}^R=: \frac{1}{\frac{1}{\beta}+r(\Omega)} \quad\text{ and }\quad \lim_{p\to \infty} \left(\lambda_{2,p}^R\right)^{{1}/{p}} = \lambda_{2,\infty}^R=: \frac{1}{s(\Omega)},
$$
where
\begin{equation*}
   s(\Omega)={\max} \left\{t: \exists x_1, x_2\in \overline{\Omega} \text{ such that }  \abs{x_1-x_2}\ge 2t, \norma{C_{i,t}}_{L^\infty(\partial\Omega)}\le\frac{1}{\beta t} \right\},
\end{equation*} 
and $C_{i,t}$ is the function
$$C_{i,t}(x)=\frac{(t-\abs{x-x_i})_+}{t}.$$

Different cases, as the mixed Robin-Dirichlet boundary conditions or the Steklov boundary conditions, were studied in \cite{gloria, stek, RS}.

In all the aforementioned papers, the geometry of the set $\Omega$ is crucial in studying the limiting problem, so one might expect a similar behavior for problem  \eqref{rel}. This idea is supported also by the fact that, for $p\to 1$, the limit eigenvalue is linked to the Cheeger constant of the set, as showed in \cite{frangianross, della2022behavior}.

\vspace{3mm}
Surprisingly, when dealing with a negative boundary condition, the limiting problem as $p$ goes to infinity appears to be completely independent of the geometry of the set, and this is what is stated in our first main result. 

\begin{teorema}
\label{existence_limit}
Let $\Omega$ be an open, bounded, Lipschitz set and let $\Set{\lambda_{p,\beta}(\Omega)}_{p>1}$ be the sequence of the first eigenvalues of the $p$-Laplace operator with Robin boundary condition defined in \eqref{rel}. Then,
\begin{equation}
\label{limit_autoval}
    \lim_{p\to \infty} \left(-\lambda_{p,\beta}(\Omega)\right)^{\frac{1}{p}}=\beta.
\end{equation}

Moreover, if $\Set{u_p}_{p>1}$ is the sequence of positive eigenfunctions associated to $\{\lambda_{p,\beta}(\Omega)\}_{p>1}$, then there exists a function $u_\infty\in W^{1,\infty}(\Omega)$ such that, up to a subsequence,
\begin{align*}
   & u_p \to u_{\infty} & \text{ uniformly in } \, \Omega, \\
   & \nabla u_p\to \nabla u_\infty & \text{ weakly in } \, L^m(\Omega),\; \forall m.
   \end{align*}
\end{teorema}

As a consequence of Theorem \ref{existence_limit}, one can deduce that Bareket's conjecture becomes trivial when $p$ goes to $+\infty$.

\vspace{3mm}
The second result we prove is that the limit function $u_{\infty}$ is a viscosity solution to a PDE problem involving the infinity Laplace operator

$$\Delta_\infty u = \left\langle D^2 u \cdot \nabla u, \nabla u \right\rangle,$$
as often happens in the limiting problems as $p$ goes to $+\infty$. 
The result is stated in the following theorem.
\begin{teorema}\label{teo1.2}
Let $u_\infty \in W^{1,\infty}(\Omega)$ be the function given in Theorem \eqref{existence_limit}. Then $u_\infty$ is a viscosity solution to
\begin{equation}
    \label{prob:lim}
    \begin{cases}
    -\min\Set{\abs{\nabla u} - \beta u,  \Delta_\infty u}=0 &\text{ in } \Omega,\\
    \min\Set{\abs{\nabla u} - \beta u, \displaystyle{\frac{\partial u}{\partial \nu}} }=0 &\text{ on } \partial \Omega.
\end{cases}
\end{equation}
\end{teorema}

Now we want to show that $\beta$ is the first eigenvalue of \eqref{prob:lim}, that is the only $\lambda>0$,  for which problem
\eqref{prob:lim}
admits a non-trivial, positive eigenfunction. The last theorem is the following

\begin{teorema}\label{teo1.3}
    Let $\Omega$ be a $C^2$ bounded domain in $\R^n$. If, for some $\lambda>0$, problem 
    \begin{equation}
    \label{lambda}
    \begin{cases}
    -\min\Set{\abs{\nabla u} - \lambda u,  \Delta_\infty u}=0 &\text{ in } \Omega,\\
    \min\Set{\abs{\nabla u} - \beta u, \displaystyle{\frac{\partial u}{\partial \nu}} }=0 &\text{ on } \partial \Omega.
\end{cases}
\end{equation}

    admits a nontrivial, positive eigenfunction $u$, then $\beta= \lambda$.
\end{teorema}

\section{Notations and Preliminaries}
\label{section_notion}
Throughout this article, $|\cdot|$ will denote the Euclidean norm in $\mathbb{R}^n$, and $\mathcal{H}^k(\cdot)$, for $k\in [0,n)$, will denote the $k-$dimensional Hausdorff measure in $\mathbb{R}^n$.

\subsection{ The viscosity \texorpdfstring{$\infty$}{inf}-eigenvalue problem}

We start by giving the notion of viscosity solution to a PDE problem, see \cite{CIL} for more details.

\begin{definizione}
  Let us consider the following boundary value problem 
    \begin{equation}
    \label{def:visc}
        \begin{cases}
            F(x, u, \nabla u, D^2u) =0 &\text{ in }\Omega, \\
            G(x,u,\nabla u)=0 &\text{ on }\partial \Omega,
        \end{cases}
    \end{equation}
    where $F:\R^n\times\R\times\R^n\times\R^{n\times n}\to \R$ and $B:\R^n\times \R\times \R^n\to\R$ are two continuous functions.
    \begin{description}
        \item[Viscosity supersolution] A lower semi-continuous function $u$ is a viscosity supersolution to \eqref{def:visc} if, whenever we fix $x_0\in \overline{\Omega}$, for every $\phi \in C^2(\overline{\Omega})$ such that $u(x)-\phi(x)\ge u(x_0)-\phi(x_0)=0$, i.e. $x_0$ is a strict minimum in $\Omega$ for $u - \phi$, then
        \begin{itemize}
            \item if $x_0 \in \Omega$, the following holds
            $$
            F\left(x_0,\phi(x_0), \nabla\phi(x_0) ,D^2\phi(x_0)\right) \geq 0
            $$
            \item if $x_0 \in \partial \Omega$, the following holds
            $$
            \max\Set{F\left(x_0,\phi(x_0),\nabla\phi(x_0) ,D^2\phi(x_0)\right), G\left(x_0, \phi(x_0), \nabla\phi(x_0)\right)}\geq 0
            $$
        \end{itemize}
        \item[Viscosity subsolution] An upper semi-continuous function $u$ is a viscosity subsolution to \eqref{def:visc} if,  whenever we fix $x_0\in \overline{\Omega}$, for every $\phi \in C^2(\overline{\Omega})$ such that $u(x)-\phi(x)\le u(x_0)-\phi(x_0)=0$, i.e. $x_0$ is a strict maximum in $\Omega$ for $u - \phi$, then
        \begin{itemize}
            \item if $x_0 \in \Omega$, the following holds
            $$
            F\left(x_0,\phi(x_0), \nabla\phi(x_0) ,D^2\phi(x_0)\right) \leq 0
            $$
            \item if $x_0 \in \partial \Omega$, the following holds
            $$
            \min\Set{F\left(x_0,\phi(x_0), \nabla\phi(x_0) ,D^2\phi(x_0)\right), G\left(x_0, \phi(x_0), \nabla\phi(x_0)\right)}\leq 0
            $$
        \end{itemize}
         \item[Viscosity solution] A continuous function $u$ is a viscosity solution to \eqref{def:visc} if it is both a supersolution and subsolution.
    \end{description}
\end{definizione}

A comparison result for supersolution and subsolution to a PDE problem holds, as proved in \cite{CIL}. It will be crucial in proving our main Theorem \ref{teo1.3}.
\begin{teorema}
    \label{comparison_cil}
    Let $\Omega$ be a bounded, open subset in $\R^n$, $F \in C(\Omega \times \R \times \R^n \times \mathcal{S}(n))$ be proper,
    
    \begin{equation}
        \label{proper}
        F(x,r,p,X)\le F(x,s,p,Y), \quad \text{ whenever } r\le s \text{ and } Y\le X
    \end{equation}
    and satisfy 
    \begin{enumerate}
        \item $\exists \gamma >0$ such that 
        $$\gamma(r-s)\leq F(x,r,p, X)- F(x,s,p, X), \quad \textrm{for}\; r\geq s, \; (x,p,X) \in \overline{\Omega} \times \R^n \times \mathcal{S}(n);$$
        \item $\exists \omega: [0,+\infty]\rightarrow [0,+\infty]$, that satisfies $\omega(0)=0$ and $\alpha >0$, such that 
        $$F(y,r,\alpha(x-y), Y)-F(x,r,\alpha(x-y), X)\leq \omega(\alpha \abs{x-y}^2+\abs{x-y}),$$
        whenever $x, y\in\Omega$, $r\in \R$, $X,Y \in \mathcal{S}(n)$ and it holds,
$$-3\alpha
        \begin{pmatrix}
            I & 0\\
            0 & I
        \end{pmatrix}  
       \leq   \begin{pmatrix}
            X & 0\\
            0 & -Y
        \end{pmatrix} \leq 3 \alpha
         \begin{pmatrix}
            I & -I\\
            -I & I
        \end{pmatrix}.$$
    \end{enumerate}
    Let $u$ be an upper semicontinuous function in  $\overline{\Omega}$ and a supersolution to  $F=0$ in $\Omega$,  and let $v$ be a lower semicontinuous function and a subsolution to $F=0$ in $\Omega$. Then, $$\max _{\Omega}(v-u )=\max_{\partial \Omega}(v-u) .$$
   
\end{teorema}

The following proposition clarifies the link between weak solutions and viscosity solutions.

\begin{prop}
\label{teorema1.3}
A continuous weak solution $u$ to \eqref{autoval_problem} is a viscosity solution to \eqref{autoval_problem}.
\end{prop}
\begin{proof}
Let $u$  be a continuous weak solution to \eqref{autoval_problem}, and let us prove that $u$ is a viscosity supersolution to \eqref{autoval_problem}.

Let $x_0 \in \Omega$, and let us consider a function $\phi $ such that $\phi(x_0)=u(x_0)$ and such that $u-\phi$ has a strict minimum at $x_0$, we claim that
\begin{equation}\label{caso1}
  -\Delta_p \phi(x_0)- \lambda_{p, \beta}(\Omega) \abs{\phi(x_0)}^{p-2}\phi(x_0)\ge 0.
\end{equation}

By contradiction, let us assume that \eqref{caso1} is false, then there exists $B_r(x_0)\subseteq\Omega$ such that

\begin{equation}
    -\Delta_p \phi(x)- \lambda_{p, \beta}(\Omega) \abs{\phi(x)}^{p-2}\phi(x)<0, \quad \forall x \in B_r(x_0),
\end{equation}
so we can define

$$m=\inf_{\abs{x-x_0}=r}u-\phi, \quad \psi=\phi+\frac{m}{2}.$$
The function $\psi$ is such that
$$\psi(x_0)>u(x_0), \quad \psi < u \text{ on } \partial B_r(x_0),$$
and if we choose $r$ small enough, we have
\begin{equation}\label{9}
    -\Delta_p \psi(x)- \lambda_{p, \beta}(\Omega) \abs{\psi(x)}^{p-2}\psi(x)<0 \quad \forall x \in B_r(x_0),
\end{equation}
so if we use $(\psi-u)_+$ in the weak formulation of \eqref{autoval_problem}, we get

\begin{equation*}
           \int_{\psi > u}  \abs{\nabla u}^{p-2} \nabla u \nabla (\psi - u)_+\, dx =\lambda_{p, \beta}(\Omega) \int_{\psi > u } \abs{u}^{p-2} u (\psi - u)_+ \, dx ,
\end{equation*}
while, if we multiply \eqref{9} by $(\psi-u)_+$ and we integrate over $\Omega$, we get
\begin{equation*} 
           \int_{\psi > u}  \abs{\nabla \psi}^{p-2} \nabla \psi \nabla (\psi - u)_+\, dx <\lambda_{p, \beta}(\Omega) \int_{\psi > u } \abs{\psi}^{p-2} \psi (\psi - u)_+ \, dx .
\end{equation*}

This implies

\begin{equation*}
    \begin{aligned}
  C(n,p) \int_{\psi > u} \abs{\nabla \psi - \nabla u}^p \, dx &\leq \int_{\psi > \psi} \left\langle \abs{\nabla \psi}^{p-2} \nabla \psi - \abs{\nabla u}^{p-2} \nabla u, \nabla(\psi - u)\right\rangle \, dx\\  
  &< \lambda_{p, \beta}(\Omega) \int_{\psi > u} \left(\abs{\psi}^{p-2} \psi-\abs{u}^{p-2} u\right) (\psi - u) \, dx <0,
\end{aligned}
\end{equation*}
and this is a contradiction.

Now, let us consider $x_0\in \partial \Omega$, and let us consider a function $\phi $ such that $\phi(x_0)=u(x_0)$ and such that $u-\phi$ has a strict minimum at $x_0$, we claim that

$$\max\left\{-\Delta_p \phi(x_0)- \lambda_{p, \beta}(\Omega) \abs{\phi(x_0)}^{p-2}\phi(x_0), \abs{\nabla\phi(x_0)}^{p-2}\frac{\partial\phi(x_0)}{\partial \nu}-\beta^p \abs{\phi(x_0)}^{p-2}\phi(x_0)\right\}\ge 0.$$

As before, let us assume by contradiction that there exists $B_r(x_0)$ such that both terms are negative. If we choose $r$ sufficiently small, in $\overline{\Omega}\cap B_r(x_0 )$, we have
\begin{equation}\label{caso2no}
     -\Delta_p\psi(x)- \lambda_{p, \beta}(\Omega) \abs{\psi(x)}^{p-2}\psi(x) < 0
\end{equation}
      
and, in $\partial\Omega \cap B_r(x_0 )$,
$$
\abs{\nabla \psi(x)}^{p-2} \frac{\partial \psi(x)}{ \partial \nu}-\beta^p \abs{\phi(x)}^{p-2} \phi(x) <0, 
$$
where, as in the previous case,  $$\psi = \phi + \frac{m}{2} \text{ and }  \displaystyle m=\inf_{\abs{x-x_0}=r}u-\phi.$$

If we choose $(\psi-u)_+$ in the weak formulation of \eqref{autoval_problem}, we get

\begin{equation*}
    \begin{multlined}
           \int_{\psi > u } \abs{\nabla u}^{p-2} \nabla u \nabla (\psi - u)_+\, dx =\lambda_{p, \beta}(\Omega) \int_{\psi > u} \abs{u}^{p-2} u (\psi - u) \, dx + \beta^p \int_{\partial \Omega \cap  \Set{\psi >u}} \abs{u}^{p-2} u (\psi - u) \, d\mathcal{H}^{n-1},
    \end{multlined}
\end{equation*}
while, if we multiply \eqref{caso2no} by $(\psi-u)_+$ and we integrate over $\Omega$, we get

\begin{equation*}
           \int_{\psi > u } \abs{\nabla\psi}^{p-2} \nabla \psi \nabla (\psi - u)_+\, dx<\lambda_{p, \beta}(\Omega) \int_{\psi > u} \abs{\psi}^{p-2} \psi (\psi - u) \, dx + \beta^p \int_{\partial \Omega \cap  \Set{\psi >u} } \abs{\phi}^{p-2} \phi (\psi - u) \, d\mathcal{H}^{n-1}.
\end{equation*}

This implies

\begin{equation*}
    \begin{aligned}
  C(n,p) \int_{\psi > u} \abs{\nabla \psi - \nabla u}^p \, dx &\leq \int_{\psi > u} \left\langle  \abs{\nabla \psi}^{p-2} \nabla \psi-\abs{\nabla u}^{p-2} \nabla u , \nabla(\psi - u)\right\rangle \, dx\\  
  &< \lambda_{p, \beta}(\Omega) \int_{\psi > u} \left(\abs{\psi}^{p-2} \psi-\abs{u}^{p-2} u\right) (\psi - u) \, dx + \\
  &+\beta^p\int_{\partial \Omega \cap  \Set{\psi >u} } \left(\abs{\phi}^{p-2} \phi-\abs{u}^{p-2}u \right)(\psi - u) \, d\mathcal{H}^{n-1}<0,
\end{aligned}
\end{equation*}
which is absurd.\\

To prove that $u$ is a viscosity subsolution, one can proceed in an analogous way.
\end{proof}

\section{Proof of main results}

We are now in position to prove Theorem \ref{existence_limit}.
\begin{proof}[Proof of Theorem \ref{existence_limit}]
 Let us consider $w=1$ as a test function in the definition \eqref{rel}, then we obtain
 \begin{equation*}
     -\lambda_{p,\beta}(\Omega)\geq \beta^p\dfrac{P(\Omega)}{\abs{\Omega}},
 \end{equation*}
 passing to the limit for $p$ that goes to $+\infty$, we obtain
 \begin{equation*}
     \liminf_{p\rightarrow \infty}\left(-\lambda_{p,\beta}(\Omega)\right)^{\frac{1}{p}}\geq \beta.
 \end{equation*}
 Let $u_p$ be the positive eigenfunction associated to $\lambda_{p,\beta}(\Omega)$ normalized such that $\norma{u_p}_{L^p(\partial\Omega)}=1$, we can write

 \begin{equation}
 \label{beta_allap}
 \left(- \lambda_{p,\beta}(\Omega) \right)\int_{\Omega}    \abs{u_p}^p\, dx + \int_\Omega \abs{\nabla u_p}^p= \beta^p
    \end{equation}
    hence, both quantities in \eqref{beta_allap} are smaller then $\beta^p$, and in particular,

    $$\beta \left(\int_{\Omega}    \abs{u_p}^p\, dx\right)^\frac{1}{p}\le\left(- \lambda_{p,\beta}(\Omega) \right)^\frac{1}{p}\left(\int_{\Omega}    \abs{u_p}^p\, dx\right)^\frac{1}{p} \leq \beta. $$
 This implies that for every $m<p$
 \begin{align*}
     &\norma{\nabla u_p}_{L^m(\Omega)}\leq \norma{\nabla u_p}_{L^p(\Omega)}\abs{\Omega}^{\frac{1}{m}-\frac{1}{p}}\leq \beta \abs{\Omega}^{\frac{1}{m}-\frac{1}{p}};\\
     & \norma{ u_p}_{L^m(\Omega)}\leq \norma{u_p}_{L^p(\Omega)}\abs{\Omega}^{\frac{1}{m}-\frac{1}{p}}\leq \abs{\Omega}^{\frac{1}{m}-\frac{1}{p}}.
 \end{align*}

By a classical argument of diagonalization, see for instance \cite{BDM}, we can extract a subsequence $u_{p_j}$ such that
\begin{gather*}
     u_{p_j} \to u_\infty \,  \text{ uniformly}  \implies \Vert u_{p_j}\Vert_{L^{p_j}(\Omega)} \to \norma{u_{\infty}}_{L^{\infty}(\Omega)}, \\
    \nabla u_{p_j} \to \nabla u_\infty \, \text{ weakly in } \, L^m(\Omega), \, \forall m >1.
\end{gather*} 
On the other hand, if we consider that
 \begin{equation*}
      -\lambda_{p,\beta}(\Omega)=\dfrac{\beta^p-\int_{\Omega}\abs{\nabla u_p}^p\;dx}{\int_{\Omega}\abs{u_p}^p\;dx}\leq\dfrac{\beta^p}{\int_{\Omega}\abs{u_p}^p\;dx}
 \end{equation*}
 and
 $$\lim_{p\rightarrow \infty}\norma{u_p}_{L^p(\Omega)}=\norma{u_{\infty}}_{L^{\infty}(\Omega)}\geq \norma{u_{\infty}}_{L^{\infty}(\partial \Omega)}=\lim_{p\rightarrow\infty}\norma{u_p}_{L^p(\partial\Omega)}=1$$
 then, passing to the limit, we have
 \begin{equation*}
     \limsup_{p\rightarrow \infty}\left(-\lambda_{p,\beta}(\Omega)\right)^{\frac{1}{p}}\leq \beta.
 \end{equation*}
 
\end{proof}

Now we can prove Theorem \ref{teo1.2}.

\begin{proof}[Proof of Theorem \ref{teo1.2}]
We divide the proof in two steps.

\textbf{Step 1} $u_\infty$ is a viscosity supersolution.\\ Let $x_0\in \Omega$  and let $\phi\in C^2(\Omega)$ be such that $u_\infty-\phi> u_\infty(x_0)-\phi(x_0)=0$ has a strict minimum in $x_0$. We want to show
$$
-\min\Set{\abs{\nabla \phi(x_0)} - \beta \phi(x_0), \Delta_\infty \phi(x_0)}\geq 0 
$$ 
Let us notice that, by the uniform convergence of $u_p$ to $u_\infty$,  $u_p- \phi$ has a minimum in $x_p$ and $x_p \to x_0$ as $p$ goes to $\infty$. If we set $\phi_p(x) =\phi(x) + c_p$ with $c_p= u_p(x_p) - \phi(x_p)$, $\phi_p$ is admissible in the definition of viscosity supersolution to problem \eqref{autoval_problem}  and Proposition \ref{teorema1.3} implies
\begin{equation}
    \label{7}
    -\abs{\nabla \phi_p(x_p)}^{p-2} \Delta \phi_p(x_p)- (p-2) \abs{\nabla \phi_p(x_p)}^{p-4}\Delta_\infty \phi (x_p)- \lambda_{p,\beta}(\Omega) \abs{\phi_p(x_p)}^{p-2}\phi_p(x_p) \geq0.
\end{equation}
Now, let us distinguish two cases.

\textbf{Case 1}: $\abs{\nabla \phi_p(x_p)}=0$, in this case, is trivial to prove
$$-\min\Set{\abs{\nabla \phi(x_0)} - \beta \phi(x_0), \Delta_\infty \phi(x_0)}\geq 0. $$

\textbf{Case 2}: $\abs{\nabla \phi_p(x_p)}\neq0$, so we can divide by $(p-2) \abs{\nabla \phi_p(x_p)}^{p-4}$, and  we obtain
\begin{equation}
\label{f1}
   \Delta_\infty \phi_p(x_p) + \frac{\abs{\nabla \phi_p(x_p)}^{2 } \Delta \phi_p(x_p)}{ p-2} \leq \frac{ \abs{\nabla \phi_p(x_p)}^4}{(p-2) \phi_p(x_p)}	\left(\frac{(-\lambda_{p,\beta}(\Omega))^{1/p} \phi_p(x_p)}{ \abs{\nabla \phi_p(x_p)}}\right)^p.
\end{equation}

If $\Delta_\infty \phi(x_0) \le 0$, the claim follows, otherwise we can raise \eqref{f1} to the power $1/p$ and we obtain

$
\abs{\nabla \phi(x_0)}-\beta\phi(x_0)\leq 0 $.

  Then, $-\min \Set{\abs{\nabla \phi(x_0)}-\beta \phi(x_0),  \Delta_\infty \phi(x_0)}\ge 0$ and $u_\infty$ is a viscosity supersolution.
  \vspace{1.5mm}

 Let us check the boundary conditions, and let us fix  $x_0 \in \partial \Omega$, $\phi\in C^2(\overline{\Omega})$ such that  $u-\phi> u(x_0)-\phi(x_0)=0$ has a strict minimum in $x_0$, our aim is to prove that 
$$
\max\Set{-\min\Set{\abs{\nabla \phi(x_0)} -\beta \phi(x_0),  \Delta_\infty \phi(x_0)}, \, \min\Set{\abs{\nabla \phi(x_0)}- \beta \phi(x_0),  \frac{\partial \phi}{\partial \nu} (x_0)}} \geq 0.
$$
We define the function $\phi_p$ as before and we observe that if, for infinitely many $p$, $x_p\in \Omega$ and inequality \eqref{7} holds true, then we get
$$
-\min\Set{\abs{\nabla \phi(x_0)} - \beta \phi(x_0),  \Delta_\infty \phi(x_0)} \geq 0.
$$
If for infinitely many $p$, $x_p \in \partial \Omega$ the following holds true 

\begin{equation}
\label{nabla}
\abs{\nabla \phi_p(x_p)}^{p-2}\left(\frac{\partial \phi_p(x_p)}{ \partial \nu}\right)\geq \beta^p \abs{\phi_p(x_p)}^{p-2} \phi_p(x_p).
\end{equation}

Let us observe that

$$\displaystyle{ \frac{\partial \phi}{\partial \nu}(x_p) \ge 0},$$
then letting $p$ to infinity
$$\displaystyle{ \frac{\partial \phi}{\partial \nu}(x_0) \ge 0}.$$
Moreover, if we divide \eqref{nabla} by $\abs{\nabla \phi_p(x_p)}^{p-2}$, and we raise to the power $1/p$
   
$$
\left(\frac{\partial \phi_p(x_p)}{ \partial \nu} \right)^{1/p}\geq \left(\frac{\beta^p \abs{\phi_p(x_p)}^{p-2} \phi_p(x_p)}{\abs{\nabla \phi_p(x_p)}^{p-2}}\right)^{1/p}, 
$$

we get $\displaystyle{\abs{\nabla \phi(x_0)} \geq \beta \phi (x_0)}$ by sending $p\to +\infty$.
That is
$$
\min\Set{\abs{\nabla \phi(x_0)}- \beta \phi(x_0), \frac{\partial \phi}{\partial \nu} (x_0)} \geq 0.
$$

\textbf{Step 2}  $u_\infty$ is a viscosity subsolution.

Let us fix $x_0\in\Omega$, $\phi\in C^2(\Omega)$ such that $u_\infty-\phi\le u(x_0)-\phi(x_0)=0 $ has a strict maximum in $x_0$. We want to prove that 
$$-\min\left\lbrace\abs{\nabla \phi(x_0)}-\beta\phi(x_0), \Delta_\infty \phi(x_0)\right\rbrace\le 0.$$ 

As before, let us consider the test function $\phi_p$ that is admissible in the definition of viscosity solution to \eqref{autoval_problem}, we have

\begin{equation*}
-\lambda_{p,\beta}(\Omega) \phi_p^{p-1}(x_p) \le (p-2) \abs{\nabla \phi_p(x_p)}^{p-4}
\left[\frac{\abs{\nabla \phi_p(x_p)}^2\Delta\phi_p(x_p)}{p-2}
+ \Delta_\infty \phi_p(x_p)\right].
\end{equation*}
Let us observe that $\abs{\nabla \phi_p(x_p)}\neq0$, so we can divide by $(p-2) \abs{\nabla \phi_p(x_p)}^{p-4}$, raise both side to the power $1/p$ and  we obtain 
$$\beta \phi(x_0)\le \abs{\nabla \phi(x_0)},$$
moreover $\Delta_\infty \phi(x_0)\ge 0$ by taking the limit,
which shows that $u_\infty$ is a viscosity subsolution to \eqref{prob:lim}.

Similar arguments to Step 1 give us the boundary conditions for viscosity subsolution. Indeed, let us fix  $x_0 \in \partial \Omega$, $\phi\in C^2(\overline{\Omega})$ such that  $u-\phi< u(x_0)-\phi(x_0)=0$ has a strict maximum in $x_0$, our aim is to prove that 
\begin{equation}
    \label{subsol}
    \min\Set{-\min\Set{\abs{\nabla \phi(x_0)} -\beta \phi(x_0),  \Delta_\infty \phi(x_0)}, \min\Set{\abs{\nabla \phi(x_0)}- \beta \phi(x_0),  \frac{\partial \phi}{\partial \nu} (x_0)}} \leq 0.
\end{equation}

We define the function $\phi_p$ as before and we observe that if, for infinitely many $p$, $x_p\in \Omega$ and inequality \eqref{subsol} holds true, then we get
$$
-\min\Set{\abs{\nabla \phi(x_0)} - \beta \phi(x_0),  \Delta_\infty \phi(x_0)} \leq 0.
$$
If for infinitely many $p$, $x_p \in \partial \Omega$, the following holds true 
$$
\abs{\nabla \phi_p(x_p)}^{p-2} \frac{\partial \phi_p(x_p)}{ \partial \nu}- \beta^p \abs{\phi_p(x_p)}^{p-2} \phi_p(x_p)   \leq 0,
$$
then only two cases can occur:

\begin{itemize}
    \item $\displaystyle{\frac{\partial\phi}{\partial\nu}}\le 0$;
    \item $\displaystyle{\frac{\partial\phi}{\partial\nu}}> 0$, then, as before
\end{itemize}

$$
\left(\abs{\nabla \phi_p(x_p)}^{p-2}\left(  \frac{\partial \phi_p(x_p)}{ \partial \nu}\right) \right)^{1/p}\leq \left(\beta^p \abs{\phi_p(x_p)}^{p-2} \phi_p(x_p)\right)^{1/p} 
$$

and if we let $p$ to $+\infty$, we get $\displaystyle{\abs{\nabla \phi(x_0)} \leq \beta \phi (x_0)}$.
That is
$$
\min\Set{\abs{\nabla \phi(x_0)}- \beta \phi(x_0), \frac{\partial \phi}{\partial \nu} (x_0)} \leq 0.
$$
\end{proof}

\begin{oss}
    Let us observe that Theorem \ref{teo1.2} implies that $u_{\infty}$ is $\infty$-subharmonic since $$-\Delta_{\infty}u_{\infty}\leq 0$$ in the viscosity sense. By the comparison Theorem \ref{comparison_cil}, this implies that $u_{\infty}$ achieves its maximum on $\partial \Omega$. Hence, $u_{\infty}$ inherits the same property of function $u_p$, which follows from the fact that they are $p$-subharmonic (see \cite{V}).
\end{oss}

\paragraph{The radial case}
Let us consider $\Omega=B_1(0)$, the ball centered at the origin with radius one. It is well known that the solution to \eqref{autoval_problem} is a radial function $u_p(x)=u_p(\abs{x})$. In this case, equation \eqref{autoval_problem} can be rewritten as follows

$$
s^{n-1} \Delta_p u_p(s)= \frac{d}{ds}\left( s^{n-1} \abs{u_p'(s)}^{p-2} u_p'(s)\right)= (-\lambda_{p,\beta}(\Omega)) \abs{u_p(s)}^{p-2} u_p(s) s^{n-1}.
$$

If we choose $u_p$ to be a positive eigenfunction and we integrate between $0$ and $t$ we have

\begin{equation*}
    t^{n-1} (u_p'(t))^{p-1} =\int_0^t (-\lambda_{p,\beta}(\Omega)) (u_p(s))^{p-1} s^{n-1} \, ds
\end{equation*}
being $u_p'(0)=0$, hence
\begin{equation*}
    u_p(r)-u_p(0) =\int_0^r\left(\frac{1}{t^{n-1}}\int_0^t (-\lambda_{p,\beta}(\Omega)) (u_p(s))^{p-1} s^{n-1}\, ds\right)^\frac{1}{p-1}\, dt
\end{equation*}
and if we let $p\to +\infty$

\begin{equation*}
    u_\infty(r)-u_\infty(0)= \beta \int_0^r \norma{u_\infty}_{L^\infty(B_t(0))}\, dt,
\end{equation*}
where we have used \eqref{limit_autoval}. Now, observing that the function $u_\infty$ is increasing along the radii, we have

\begin{equation*}
    u_\infty(r)-u_\infty(0)= \beta \int_0^r u_\infty(t)\, dt,
\end{equation*}

and $u_\infty$ must have the form

$$u_\infty(t)=Ce^{\beta t}.$$
 \paragraph{One dimensional case} This case is analogous to the radial case, so let us consider $\Omega=[-1,1]$. It is well known that the solution to \eqref{autoval_problem} is an even function $u_p(x)=u_p(\abs{x})$. In this case, equation \eqref{autoval_problem} can be rewritten as follows.

$$
\Delta_p u_p(s)= \frac{d}{ds}\left( \abs{u_p'(s)}^{p-2} u_p'(s)\right)= (-\lambda_{p,\beta}(\Omega)) \abs{u_p(s)}^{p-2} u_p(s).
$$
If we choose $u_p$ to be a positive eigenfunction and we integrate between $0$ and $t$ we have

\begin{equation*}
   (u_p'(t))^{p-1} =\int_0^t (-\lambda_{p,\beta}(\Omega)) (u_p(s))^{p-1}  \, ds
\end{equation*}

hence,
\begin{equation*}
    u_p(r)-u_p(0) =\int_0^r\left(\int_0^t (-\lambda_{p,\beta}(\Omega)) (u_p(s))^{p-1} \, ds\right)^\frac{1}{p-1}\, dt
\end{equation*}
and if we let $p\to +\infty$

\begin{equation*}
    u_\infty(r)-u_\infty(0)= \beta \int_0^r \norma{u_\infty}_{L^\infty([0,1])}\, dt,
\end{equation*}
where we have used \eqref{limit_autoval}. Now, observing that the function $u_\infty$ is increasing along the radii, we have

\begin{equation*}
    u_\infty(r)-u_\infty(0)= \beta \int_0^r u_\infty(t)\, dt,
\end{equation*}
hence $u_\infty$ must have the form

$$u_\infty(t)=Ce^{\beta t}.$$

The analysis of the radial case suggests us to consider an auxiliary problem to prove our Theorem \ref{teo1.3}. In particular, we want to apply the comparison result Theorem \ref{comparison_cil} to the new PDE, so let us start by showing that $v=\log u$ is a viscosity solution to a PDE problem.

\begin{lemma}\label{logaritmo}
Let $\Omega$ be an open bounded set. 
    If $u$ is a positive viscosity solution to the problem 
    $$-\min\left\{\abs{\nabla \phi}-\lambda \phi, \Delta_{\infty}\phi\right\}=0\qquad \textrm{in}\; \Omega,$$
     then $v=\log{u}$ is a viscosity solution to 
    \begin{equation}
    \label{log_eq}
    -\min\left\{\abs{\nabla v}-\lambda, \Delta_{\infty}v+\abs{\nabla v}^4\right\}=0 \qquad \textrm{in }\Omega.
    \end{equation}
     
\end{lemma}
\begin{proof}
    Let $\phi \in C^2(\Omega)$, such that $v-\phi\geq v(x_0)-\phi (x_0)=0$, then $\psi=e^{\phi}$ is a good test function in the definition of viscosity supersolution for $u$ at the point $x_0$. So, we have
    $$-\min\{\abs{\nabla \psi(x_0)}-\lambda \psi (x_0),\Delta_{\infty} \psi (x_0)\}\geq 0.$$
    If we rewrite this inequality in terms of $\phi$, we get
    $$-\min\left\{e^{\phi (x_0)}\left(\abs{\nabla \phi (x_0)}-\lambda\right),e^{3\phi(x_0)}\left(\Delta_{\infty}\phi(x_0)+\abs{\nabla \phi (x_0)}^4\right)\right\}\geq 0,$$
    and so $v$ is a viscosity supersolution to the equation \eqref{log_eq}.
    The proof of the viscosity subsolution is analogous.
    
\end{proof}

We can now proceed with the proof of Theorem \ref{teo1.3}.

\begin{proof}[Proof of Theorem \ref{teo1.3}]
    Let $u$ be a positive nontrivial  eigenfunction associated to $\lambda$, then by Lemma \ref{logaritmo} $v=\log(u)$ is a viscosity solution to 

    \begin{equation}
    \label{log_lambda}
        -\min\left\{\abs{\nabla v}-\lambda, \Delta_{\infty}v+ \abs{\nabla v}^4\right\}=0, \quad \text{ in } \Omega.
    \end{equation}

Let us start by proving that $\lambda\ge \beta$. 

In \cite{charro}, it was proved that $-(\lambda+\varepsilon)d(x,\partial\Omega)$ is a viscosity solution to

$$-\min\{\abs{\nabla \phi}-(\lambda +\varepsilon), \Delta_{\infty}\phi\},$$
    hence, for every $\varepsilon>0 $, and $\gamma <\frac{\varepsilon}{2R}$, the function 
    $$g_{\varepsilon,\gamma}(x)=-(\lambda+\varepsilon)d(x,\partial \Omega)+\gamma d^2(x,\partial\Omega),$$ 
    is a viscosity subsolution to \eqref{log_lambda}.\\
    By the comparison principle in \cite{CIL}, Theorem \ref{comparison_cil}, it holds
    $$\max_{\Omega}(g_{\varepsilon,\gamma}-v)=\max_{\partial \Omega}(g_{\varepsilon,\gamma}-v)=-\log(u(x_0)).$$
    Then, $$u(x)\geq u(x_0)e^{g_{\varepsilon,\gamma}(x)}.$$

     In a tubular neighbourhood $\Gamma$ of $\partial\Omega$, the boundary $\partial\Omega$ and the distance function $d(x,\partial\Omega)$ share the same regularity (see \cite{GT}):  so both $d(x,\partial\Omega)$ and $g_{\varepsilon,\gamma}$ are $C^2(\Gamma)$, and we can use
     $\Phi(x)=u(x_0)e^{g_{\varepsilon,\gamma}(x)}$ in the definition of viscosity supersolution to \eqref{prob:lim} for $u$.\\ By a direct computation, we have that 
     
     \begin{align*}
     &\abs{\nabla \Phi(x_0)}-\lambda \Phi(x_0)=\Phi(x_0)\left(\lambda+\varepsilon-\lambda\right)=\varepsilon\Phi(x_0)>0,
     \vspace{5mm}
     \\
     &\Delta_{\infty} \Phi=\Phi^3\left(\abs{\nabla g_{\varepsilon,\gamma}}^4+\Delta_{\infty}g_{\varepsilon,\gamma}\right)=\Phi^3\left(\abs{\nabla g_{\varepsilon,\gamma}}^4+2\gamma\right)>0\\
     &\dfrac{\partial \Phi}{\partial \nu}(x_0)=\Phi(x_0)\nabla g_{\varepsilon,\gamma}(x_0)\cdot (-\nabla d(x_0,\partial \Omega))=\Phi(x_0) \abs{\nabla g_{\varepsilon,\gamma}}>0.
     \end{align*}
     Hence, 
     $$0\leq\abs{\nabla \phi(x_0)} -\beta \Phi(x_0)=\Phi(x_0)(\lambda- \varepsilon-\beta),$$
     and passing to the limit for $\varepsilon$ that goes to $0$, we obtain $\lambda \geq\beta$.
     
   Let us now prove that $\lambda\le \beta$.  
    
    For every $\varepsilon>0 $, 
    $$h_{\varepsilon}(x)=-(\lambda-\varepsilon)d(x,\partial \Omega),$$ 
    is a viscosity supersolution to \eqref{log_lambda} in a tubular neighborhood $\Gamma$ of the boundary of $\Omega$.
This is true because in $\Gamma$ the function $h_{\varepsilon}$ is differentiable, so for every $\phi \in C^2(\Gamma)$ such that $h_{\varepsilon}(x)-\phi(x)\geq h_{\varepsilon}(x_0)-\phi(x_0)=0$, 
$$\abs{\nabla \phi (x_0)}=\abs{h_{\varepsilon}(x_0)}=\lambda-\varepsilon,$$
for every $x_0 \in \Gamma$.
 
By the comparison principle in \cite{CIL} (Theorem \ref{comparison_cil}), it holds
    $$\max_{\Gamma}(v-h_{\varepsilon})=\max_{\partial \Gamma}(v-h_{\varepsilon})=\log(u(x_0))-h_\varepsilon(x_0).$$
    Then, $$u(x)\leq u(x_0)e^{h_{\varepsilon}(x)-h_\varepsilon(x_0)}.$$

      and we can use
     $\Phi(x)=u(x_0)e^{h_{\varepsilon}(x)-h_\varepsilon(x_0)}$ in the definition of viscosity subsolution to \eqref{prob:lim} for $u$. By a direct computation, we have that 

     \begin{equation}
         \abs{\nabla \Phi(x_0)}-\lambda \Phi(x_0)=\Phi(x_0)\left(\lambda-\varepsilon-\lambda\right)=-\varepsilon\Phi(x_0)<0,
     \end{equation}
     hence
$$-\min\left\{\abs{\nabla \Phi(x_0)}-\lambda \Phi(x_0), \Delta_{\infty}\phi\right\}\leq 0,$$ which means that $x_0$ is necessarily on $\partial \Omega$. Moreover,
     
     \begin{equation*}
     \dfrac{\partial \Phi}{\partial \nu}(x_0)=\Phi(x_0)\nabla h_{\varepsilon}(x_0)\cdot (-\nabla d(x_0,\partial \Omega))=\Phi(x_0) \abs{\nabla h_{\varepsilon}}>0.
     \end{equation*}
     Hence, 
     $$0\geq\abs{\nabla \phi(x_0)} -\beta \Phi(x_0)=\Phi(x_0)(\lambda- \varepsilon-\beta),$$
     and passing to the limit for $\varepsilon$ that goes to $0$, we obtain $\lambda \le\beta$.
     
\end{proof}

\subsection*{Acknowledgements}
The authors were partially supported by Gruppo Nazionale per l’Analisi Matematica, la Probabilità e le loro Applicazioni
(GNAMPA) of Istituto Nazionale di Alta Matematica (INdAM).   \\
During the preparation of this work, Alba Lia Masiello was partially supported by  PRIN 2022, 20229M52AS: "Partial differential equations and related geometric-functional
inequalities,"\\
CUP:E53D23005540006.

\Addresses 

\addcontentsline{toc}{chapter}{Bibliografia}
\bibliographystyle{plain}
\bibliography{biblio}

\begin{thebibliography}{10}

\bibitem{AMNT}
V.~Amato, A.~L. Masiello, C.~Nitsch, and C.~Trombetti.
\newblock On the solutions to {$p$}-{P}oisson equation with {R}obin boundary conditions when {$p$} goes to {$+\infty$}.
\newblock {\em Adv. Nonlinear Anal.}, 11(1):1631--1649, 2022.

\bibitem{AMNT2}
V.~Amato, A.~L. Masiello, C.~Nitsch, and C.~Trombetti.
\newblock On the second eigenvalue of the infinity laplacian with robin boundary conditions, 2024.

\bibitem{gloria}
G.~Ascione and G.~Paoli.
\newblock The orthotropic {$p$}-{L}aplace eigenvalue problem of {S}teklov type as {$p\to +\infty$}.
\newblock {\em J. Math. Anal. Appl.}, 501(2):Paper No. 125219, 26, 2021.

\bibitem{frangianross}
R.~Barbato, F.~Della~Pietra, and G.~Piscitelli.
\newblock On the first {R}obin eigenvalue of the {F}insler {$p$}-{L}aplace operator as {$p\to 1$}.
\newblock {\em J. Math. Anal. Appl.}, 540(2):Paper No. 128660, 25, 2024.

\bibitem{Bareket}
M.~Bareket.
\newblock On an isoperimetric inequality for the first eigenvalue of a boundary value problem.
\newblock {\em SIAM Journal on Mathematical Analysis}, 8(2):280--287, 1977.

\bibitem{BDM}
T.~Bhattacharya, E.~DiBenedetto, and J.~Manfredi.
\newblock Limits as {$p\to\infty$} of {$\Delta_pu_p=f$} and related extremal problems.
\newblock {\em Rend. Sem. Mat. Univ. Politec. Torino}, Special Issue:15--68 (1991), 1989.
\newblock Some topics in nonlinear PDEs (Turin, 1989).

\bibitem{charro}
Fernando Charro.
\newblock Explicit solutions of {J}ensen's auxiliary equations via extremal {L}ipschitz extensions.
\newblock {\em Electron. J. Differential Equations}, pages Paper No. 37, 6, 2020.

\bibitem{CIL}
M.~G. Crandall, H.~Ishii, and P.-L. Lions.
\newblock User's guide to viscosity solutions of second order partial differential equations.
\newblock {\em Bull. Amer. Math. Soc. (N.S.)}, 27(1):1--67, 1992.

\bibitem{della2022behavior}
F.~Della~Pietra, C.~Nitsch, F.~Oliva, and C.~Trombetti.
\newblock On the behavior of the first eigenvalue of the $p$-{L}aplacian with {R}obin boundary conditions as p goes to 1.
\newblock {\em Advances in Calculus of Variations}, 2022.

\bibitem{EKNT}
L.~Esposito, B.~Kawohl, C.~Nitsch, and C.~Trombetti.
\newblock The {N}eumann eigenvalue problem for the {$\infty$}-{L}aplacian.
\newblock {\em Atti Accad. Naz. Lincei Rend. Lincei Mat. Appl.}, 26(2):119--134, 2015.

\bibitem{Freitas_Krejcirik}
P.~{Freitas} and D.~{Krejcirik}.
\newblock The first {R}obin eigenvalue with negative boundary parameter.
\newblock {\em Advances in Mathematics}, 280:322--339, 2015.

\bibitem{stek}
J.~Garcia-Azorero, J.~J. Manfredi, I.~Peral, and J.~D. Rossi.
\newblock Steklov eigenvalues for the {$\infty$}-{L}aplacian.
\newblock {\em Atti Accad. Naz. Lincei Rend. Lincei Mat. Appl.}, 17(3):199--210, 2006.

\bibitem{GT}
D.~Gilbarg and N.~S. Trudinger.
\newblock {\em Elliptic partial differential equations of second order}.
\newblock Classics in Mathematics. Springer-Verlag, Berlin, 2001.
\newblock Reprint of the 1998 edition.

\bibitem{Lin_ju_lambda2}
P.~Juutinen and P.~Lindqvist.
\newblock On the higher eigenvalues for the {$\infty$}-eigenvalue problem.
\newblock {\em Calc. Var. Partial Differential Equations}, 23(2):169--192, 2005.

\bibitem{JLM}
P.~Juutinen, P.~Lindqvist, and J.~J. Manfredi.
\newblock The {$\infty$}-eigenvalue problem.
\newblock {\em Arch. Ration. Mech. Anal.}, 148(2):89--105, 1999.

\bibitem{kovavrik2017p}
H.~Kova{\v{r}}{\'\i}k and K.~Pankrashkin.
\newblock On the $p$-{L}aplacian with {R}obin boundary conditions and boundary trace theorems.
\newblock {\em Calculus of Variations and Partial Differential Equations}, 2(56):1--29, 2017.

\bibitem{paoli2019two}
G.~Paoli and L.~Trani.
\newblock Two estimates for the first {R}obin eigenvalue of the {F}insler {L}aplacian with negative boundary parameter.
\newblock {\em Journal of Optimization Theory and Applications}, 181(3):743--757, 2019.

\bibitem{RS}
J.~D. Rossi and N.~Saintier.
\newblock The limit as {$p\to+\infty$} of the first eigenvalue for the {$p$}-{L}aplacian with mixed {D}irichlet and {R}obin boundary conditions.
\newblock {\em Nonlinear Anal.}, 119:167--178, 2015.

\bibitem{RS2}
J.~D. Rossi and N.~Saintier.
\newblock On the first nontrivial eigenvalue of the {$\infty$}-{L}aplacian with {N}eumann boundary conditions.
\newblock {\em Houston J. Math.}, 42(2):613--635, 2016.

\bibitem{segura}
J.~C. Sabina~de Lis and S.~Segura~de Le\'on.
\newblock Higher {R}obin eigenvalues for the {$p$}-{L}aplacian operator as {$p$} approaches 1.
\newblock {\em Calc. Var. Partial Differential Equations}, 63(7):Paper No. 169, 47, 2024.

\bibitem{V}
J.~L. V\'{a}zquez.
\newblock A strong maximum principle for some quasilinear elliptic equations.
\newblock {\em Appl. Math. Optim.}, 12(3):191--202, 1984.

\end{thebibliography}

\end{document}